\documentclass[letterpaper, 10pt, conference]{ieeeconf}
\IEEEoverridecommandlockouts

\usepackage{lipsum}
\usepackage{cite}
\usepackage{amsmath,amssymb,amsfonts}
\usepackage{algorithm,algorithmic}
\usepackage{hyperref}
\hypersetup{hidelinks=true}
\usepackage{textcomp}
\makeatother

\usepackage{graphicx}
\usepackage{amsthm}
\usepackage{tabu}
\usepackage{breqn}
\usepackage{verbatim}
\usepackage{esvect}
\usepackage{footnote}
\usepackage{siunitx}
\usepackage{color}
\usepackage{url}
\usepackage[makeroom]{cancel}
\usepackage{todonotes}
\usepackage{subcaption}

\newtheorem{theorem}{Theorem}

\newtheorem{definition}{Definition}
\newtheorem{lemma}[theorem]{Lemma}
\newtheorem{proposition}[theorem]{Proposition}
\newtheorem{remark}{Remark}
\newtheorem*{problem statement}{Problem Statement}
\newtheorem{problem}{Problem}

\newcommand{\real}{\mathbb{R}}

\newcommand{\Cc}{{\mathcal{C}}}

\newcommand{\Kc}{{\mathcal{K}}}

\newcommand{\bS}{{\mathbf{S}}}
\newcommand{\bx}{{\mathbf{x}}}

\newcommand{\bX}{{\mathbf{X}}}

\newcommand{\ba}{{\mathbf{a}}}

\newcommand{\bg}{{\mathbf{g}}}

\newcommand{\bu}{{\mathbf{u}}}

\newcommand{\bv}{{\mathbf{v}}}
\newcommand{\bc}{{\mathbf{c}}}

\newcommand{\bd}{{\mathbf{d}}}
\newcommand{\bk}{{\mathbf{k}}}

\newcommand{\bbf}{{\mathbf{f}}}
\newcommand{\bpsi}{{\boldsymbol{\psi}}}
\newcommand{\bPsi}{{\boldsymbol{\Psi}}}

\newcommand{\bmu}{{\boldsymbol{\mu}}}
\newcommand{\bell}{{\boldsymbol{\ell}}}

\newcommand{\bA}{{\mathbf{A}}}
\newcommand{\bB}{{\mathbf{B}}}

\newcommand{\bG}{{\mathbf{G}}}
\newcommand{\bK}{{\mathbf{K}}}

\newcommand{\bI}{{\mathbf{I}}}

\newcommand{\bL}{{\mathbf{L}}}

\newcommand{\bM}{{\mathbf{M}}}
\newcommand{\bP}{{\mathbf{P}}}

\newcommand{\bC}{{\mathbf{C}}}

\newcommand{\bQ}{{\mathbf{Q}}}
\newcommand{\bR}{{\mathbf{R}}}
\newcommand{\bY}{{\mathbf{Y}}}
\newcommand{\bT}{{\mathbf{T}}}

\newcommand{\bGamma}{{\boldsymbol{\Gamma}}}
\newcommand{\bPhi}{{\boldsymbol{\Phi}}}
\newcommand{\bzero}{{\boldsymbol{0}}}

\newcommand\xqed[1]{%
  \leavevmode\unskip\penalty9999 \hbox{}\nobreak\hfill
  \quad\hbox{#1}}
\newcommand\demo{\xqed{$\bullet$}}

\newcommand{\longthmtitle}[1]{\mbox{}\emph{(#1):}}
\newcommand{\setdef}[2]{\{#1 : #2\}}

\newcommand{\norm}[1]{\left\lVert#1\right\rVert}

\allowdisplaybreaks
\renewcommand{\baselinestretch}{.98}

\begin{document}
\title{\LARGE \bf Safe Stabilizing Linear Feedback: \\ Necessary and Sufficient Conditions, Optimality, and Margins}
\author{Pol Mestres, Shima Sadat Mousavi, Pio Ong, and Aaron D. Ames
\thanks{The authors are with the Department of Mechanical and Civil
Engineering, California Institute of Technology, Pasadena, CA 91125, USA.
Emails: \texttt{mestres,smousavi,pioong,ames@caltech.edu}.
This research is supported by The Boeing Company.}
}
\maketitle

\begin{abstract}
Control barrier functions (CBFs) have become an important controller design tool for autonomous systems subject to safety constraints. Despite their popularity, recent works have shown that CBF-based controllers can destabilize the internal dynamics of the system.
In this paper, we consider linear systems with affine safety constraints and design linear feedback controllers that satisfy high-order CBF (HOCBF) constraints while rendering the origin globally exponentially stable.
We first characterize the exact class of all linear gain matrices that globally satisfy the HOCBF constraints, including necessary and sufficient conditions for when this class is nonempty. Then, by leveraging the recently introduced notion of CBF output dynamics and CBF internal dynamics, we provide the necessary and sufficient conditions for the existence of stabilizing gain matrices within that class. 
Finally, we show that Linear Quadratic Regulator (LQR) and robust control problems can be solved  while being constrained within this class of safe and stabilizing gain matrices, through standard linear control techniques such as algebraic Riccati equations (AREs) and Linear Matrix Inequalities (LMIs). 
We illustrate our results in a simulation example.
\end{abstract}

\section{Introduction}
Complex autonomous systems such as humanoid robots or aerospace vehicles are subject to strict safety constraints.
Control barrier functions (CBFs)~\cite{ADA-SC-ME-GN-KS-PT:19} have become a popular tool to design controllers that satisfy such safety constraints, and have been applied to a wide range of tasks including humanoid walking~\cite{SH-XX-ADA:15} or aircraft flight~\cite{MM-EL:26,AWS-MHC-TGM-ADA:26}.

However, recent works~\cite{JJC-CJT-SS-KS:25,LB-SZ-APS:26,NM-JC-PS-KZ:25,PM-YC-EDA-JC:26-jnls,PM-SSM-ADA:26,SSM-PM-ADA:26} have shown that controllers designed to satisfy CBF constraints can destabilize the system dynamics.
More concretely,~\cite{NM-JC-PS-KZ:25,PM-YC-EDA-JC:26-jnls,PM-SSM-ADA:26,SSM-PM-ADA:26} show that these instabilities can arise even in the case where the controller is designed through so-called \textit{CBF-based safety filters}, which minimally modify a nominal stabilizing controller to satisfy the CBF constraints.
Hence, despite their guaranteed safety, this lack of stability guarantees compromises the use of CBF-based controllers altogether.

Although the works~\cite{ZM-FB-AGA:24,PM-SSM-ADA:26,SSM-PM-ADA:26} provide conditions for stability of the closed-loop system obtained from CBF-based controllers, these can be conservative, specially for multiple CBF constraints. 
In general, the nonlinear nature of the filtered controller significantly complicates the analysis of the closed-loop dynamical properties and makes a full characterization of the stability properties of safety filters an open problem (even for linear systems and affine CBFs).

On the other hand, the nonlinear controllers obtained from CBF-based safety filters pose a difficulty when analyzing robustness metrics for linear systems such as phase and gain margins, since those are often defined for linear controllers~\cite{SSM-PM-ADA:26-robustness}.
This presents a difficulty when implementing safety filters in applications where such robustness metrics need to be certified, such as flight control systems, where they are increasingly being implemented~\cite{MM-EL:26,AWS-MHC-TGM-ADA:26}.

Motivated by the difficulty of studying the stability properties of safety filters and the nonlinear nature of their induced controllers, 
in this paper we seek to find \textit{linear} controllers that simultaneously certify safety and stability.
To do so, we follow an approach closely aligned with the one in~\cite{JJC-CJT-SS-KS:25}, which draws inspiration from the input-output (IO) linearization literature (cf.~\cite[Chapter 5]{AI:95}) and considers the CBF as an output of the system. By doing so, the dynamics can be rewritten as the union of the so-called \textit{CBF output dynamics} and \textit{CBF internal dynamics}.
This analysis shows that the stability of the CBF internal dynamics 
guarantees the stability of the full nonlinear control system under a CBF-based controller.

The contributions of this paper are as follows. First, we 
characterize the set of linear controllers that satisfy a set of affine CBF constraints for a linear system.
Second, we specialize the analysis done in~\cite{JJC-CJT-SS-KS:25} for linear systems and affine CBFs. 
In this case, we show that the CBF internal dynamics are linear, and we derive a characterization of their stabilizability. 
This analysis reveals necessary and sufficient conditions for the existence of linear stabilizing controllers satisfying a set of CBF constraints, and provides a characterization of all such controllers.
Third, we use this characterization to solve a Linear Quadratic Regulator (LQR) problem to find the optimal controller within this set of safe and stabilizing controllers, and show that the solution can be obtained through an algebraic Riccati equation (ARE).
Fourth, we leverage robust control techniques to obtain linear CBF-constrained controllers with provable gain and phase margin guarantees through a set of convex constraints.~\footnote{Notation: Throughout the paper, we denote by $\mathbb{N}$, $\mathbb{Z}_{>0}$ $\real$, $\real_{\geq0}$ the set of natural, positive integer, real, and non-negative real numbers, respectively. We use bold (respectively, non-bold) symbols to represent vectors (respectively, scalars).
Given $n\in\mathbb{N}$, we let $\mathbf{0}_n$ be the zero vector in $\real^n$ and write $[n] = \{ 1, 2, \hdots, n \}$. A function $\alpha:\real\to\real$ is of extended class $\Kc$ if it is continuous, strictly increasing, and satisfies $\alpha(0) = 0$.
Given $\ba_1, \hdots, \ba_m \in \real^n$, $\bA = [\ba_1; \hdots; \ba_m]\in\real^{m\times n}$ denotes the matrix whose $i$-th row is $\ba_i$ (for $i\in[m]$).
We denote by $\mathbb{S}_{++}^n\subset\real^{n\times n}$ the set of $n\times n$ symmetric positive definite matrices.
Given a complex-valued matrix $\tilde{\bC}\in\mathbb{C}^{n\times n}$, $\sigma(\tilde{\bC})$ denotes the maximum singular value of $\tilde{\bC}$.
Given $\bG:\mathbb{C}\to\mathbb{C}^{n\times n}$, we write $\norm{\bG}_{\infty} := \sup\limits_{\omega\in\real} \sigma( \bG(j\omega) )$.
Let $k\in\mathbb{Z}_{>0}$, $\{ n_i\in\mathbb{Z}_{>0} \}_{i=1}^k$, and $\bA_i \in \real^{n_i \times n_i}$ for $i\in[k]$.
Further let $\bar{n} = \sum_{i=1}^k n_i$.
Then, $\bar{\bA} = \text{blkdiag}(\{ \bA_i \}_{i=1}^k) \in \real^{\bar{n}\times\bar{n}}$ is the block diagonal matrix formed  by aligning the matrices $\{ \bA_i \}_{i=1}^k$ along the diagonal of $\bar{\bA}$.
}

\section{Background}

\subsection{Control barrier functions for control-affine systems}

Consider a control-affine system:
\begin{align}\label{eq:control-affine-system}
    \dot{\bx} = \bbf(\bx) + \bg(\bx) \bu,
\end{align}
where $\bx\in\real^n$ is the state, $\bu\in\real^m$ the control input, and $\bbf:\real^n\to\real^n$, $\bg:\real^n\to\real^{n\times m}$ are sufficiently smooth.
We consider $p\in\mathbb{N}$ safety constraints defined by sufficiently smooth functions $\{ h_i:\real^n\to\real \}_{i=1}^p$ as:
\begin{equation}\label{eq:safety_constraint}
    \Cc = \setdef{\bx\in\real^n}{h_i(\bx) \geq 0, \ \forall i\in[p]}.
\end{equation}
Ideally, we would like to enforce safety for system~\eqref{eq:control-affine-system} by designing a controller that renders $\Cc$ forward invariant. However, depending on the system dynamics and the constraint functions, rendering all of $\Cc$ forward invariant may not be possible. One concept that helps formalize this limitation is the \textit{relative degree} of the constraint function.
\begin{definition}\longthmtitle{Relative degree}
    A sufficiently smooth function $h:\real^n\to\real$ has relative degree $r\in\mathbb{Z}_{>0}$ on a set $\mathcal{D}$ with respect to system~\eqref{eq:control-affine-system} if, for each $\bx\in\mathcal{D}$, $L_{\bg}L_{\bbf}^{r-1}h(\bx) \neq \mathbf{0}_m^\top$  and $L_{\bg}L_{\bbf}^k h(\bx) = \mathbf{0}_m^\top$ for all $k < r-1$.~\hfill$\diamond$
\end{definition}
The high-order control barrier function (HOCBF) framework provides a tool for enforcing safety constraints with an arbitrary relative degree. The following HOCBF construction produces a control invariant (when there are no input bounds) set $\bar\Cc_i$  associated with each constraint~$h_i$.


\begin{theorem}\longthmtitle{HOCBF~\cite[Thm. 4]{WX-CB:22}}\label{thm:hocbf}
    Consider the control-affine system~\eqref{eq:control-affine-system} with a safety constraint $\Cc$ in~\eqref{eq:safety_constraint}.
    Assume that, for each $i\in[p]$, $h_i$ has relative degree ${r_i\in\mathbb{N}}$ on $\real^n$. Given extended class-$\Kc$ functions $\{ \bar{\alpha}_{i,j} \}_{ i\in[p], j\in[r_i] }$, define recursively from $h_{i,0}(\bx)\triangleq h_i(\bx)$, for each ${j\in[r_i-1]}$:
    \begin{align}\label{eq:HOCBF}
        h_{i,j}(\bx)\triangleq L_\bbf{h}_{i,j-1}(\bx)+\bar{\alpha}_{i,j}(h_{i,j-1}(\bx)).
    \end{align}
    Then, with $\Cc_{i,j} \triangleq \setdef{\bx\in\real^n}{h_{i,j}(\bx) \geq 0}$,  their set intersection  $\bar{\Cc}_i \triangleq \bigcap_{j=0}^{r_i-1} \Cc_{i,j}$ is control invariant. Furthermore,
    %
    any locally Lipschitz controller $\bk:\real^n\to\real^m$ simultaneously satisfying HOCBF constraints:
    \begin{align}\label{eq:hocbf-inequality}
        L_{\bbf}h_{i,r_i-1}(\bx) + L_{\bg}h_{i,r_i-1}(\bx) \bk(\bx) + \bar{\alpha}_{i,r_i}(h_{i,r_i-1}(\bx)) \geq 0,
    \end{align}
    for all $i\in[p]$ and $\bx\in\bar{\Cc} \triangleq \bigcap_{i=1}^p \bar{\Cc}_i\subseteq\Cc$ renders $\bar{\Cc}$ forward invariant for the closed-loop system with feedback $\bu = \bk(\bx)$.
\end{theorem}

%
%
%
In general nonlinear settings, designing controllers that satisfy all HOCBF constraints~\eqref{eq:hocbf-inequality} can be challenging, especially when stability is also considered, see \cite{PM-YC-EDA-JC:26-jnls,MFR-APA-PT:21,LB-SZ-APS:26,NM-JC-PS-KZ:25,PM-SSM-ADA:26}.

\subsection{Linear systems with affine constraints}

Given the difficulty of designing safe and stabilizing controllers for general nonlinear systems, we focus throughout the rest of the paper on linear systems:
\begin{align}\label{eq:linear-system}
    \dot{\bx} = \bA \bx + \bB \bu,
\end{align}
with $\bA \in \real^{n\times n}$ and $\bB\in\real^{n\times m}$. We further assume that $h_i(\bx) = \bc_i^\top \bx + d_i$ is affine in the state $\bx$, with $\bc_i \in \real^n$ and $d_i \in \real$, and that $h_i$ has relative degree $r_i$.

%
%
In this setting, the stability of the equilibrium $\bx^\star=\bzero$ is well understood under linear state-feedback $\bu=-\bK\bx$, i.e., the origin is globally exponentially stable (GES) if and only if $\bA-\bB\bK$ is Hurwitz. Motivated by this linear feedback characterization, we choose the functions $\bar{\alpha}_{i,j}(y) = \alpha_{i,j} y$ to be linear with slopes $\alpha_{i,j} > 0$ so that the intermediate auxiliary functions~\eqref{eq:HOCBF} are affine in the state $\bx$. In particular, defining $\bPhi_{i,j}(\bA) = \prod_{l=1}^j (\bA+\alpha_{i,l} \bI_n)$, we have
$h_{i,j}(\bx) = \bc_i^\top \bPhi_{i,j}(\bA) \bx + d_i \prod_{l=1}^j \alpha_{i,l}$~\cite[Lemma 2]{PM-SSM-ADA:26}.
Furthermore, the HOCBF constraint~\eqref{eq:hocbf-inequality} reads as:
\begin{align}\label{eq:hocbf-linear}
    \bc_i^\top \bPhi_i(\bA) \bx + \bc_i^\top \bA^{r_i-1} \bB \bu + \alpha_i d_i \geq 0, \quad i\in[p].
\end{align}
where $\bPhi_i \triangleq \bPhi_{i,r_i}$ and $\alpha_i \triangleq \prod_{j=1}^{r_i} \alpha_{i,j}$ are introduced for compactness of the presentation.
%
%
%
With the HOCBF constraints, we seek to solve the following problem.

\begin{problem}\label{prob:problem-stability}
    Characterize necessary and sufficient conditions under which there exists a gain matrix $\bK\in\real^{m\times n}$ such that the linear feedback $\bu=-\bK\bx$ satisfies HOCBF constraints~\eqref{eq:hocbf-linear} 
    for all $\bx\in\real^n$
    while rendering the origin GES.
\end{problem}

Although Theorem~\ref{thm:hocbf} only requires the HOCBF constraints to be satisfied in $\bar{\Cc}$ to ensure its forward invariance, Problem~\ref{prob:problem-stability} imposes the constraints globally (for all $\bx\in\real^n$).
This stronger requirement is motivated by the fact that HOCBF constraints outside $\bar{\Cc}$ guarantee asymptotic stability of $\bar{\Cc}$ (cf.~\cite[Proposition 3]{XT-WSC-DVD:22}), as well as a robust notion of safety, referred to as input-to-state safety~\cite{SK-ADA:19}.
Additionally, the requirement of global satisfaction of the HOCBF conditions makes the controllers globally defined and significantly simplifies the results in the paper.

\subsection{Optimal and robust linear controllers}\label{sec:optimal-robust-linear-safe}

Provided that Problem~\ref{prob:problem-stability} admits a non-empty class of linear controllers, we next consider how to select among them according to performance and robustness criteria. A natural performance metric is the trajectory quadratic cost:
$$
J(\bx_0,\bk) = \int_0^\infty    \Big(
        \bx(t)^\top \bQ \bx(t)  +  \bk(\bx(t))^{\top}  \bR \bk(\bx(t))
        \Big) dt
$$
evaluated over the trajectory $t\mapsto \bx(t)$ of~\eqref{eq:linear-system} from the initial condition $\bx_0\in\real^n$ under feedback policy $\bu=\bk(\bx)$, with $\bQ \in \mathbb{S}_{++}^n$ and $\bR \in \mathbb{S}_{++}^m$. In the unconstrained LQR problem, there exists a single linear feedback policy that minimizes this cost for every initial condition. However, when constraining the controller to satisfy the HOCBF constraints~\eqref{eq:hocbf-linear}, an optimal safe feedback policy $\bk$ need not be linear, and even if it is constrained to be linear, it is unclear whether the optimal solution is independent of the initial condition $\bx_0$.
\begin{problem}\label{prob:optimality}
    Given $\{ \alpha_{i,j} \}_{i\in[p], j\in[r_i]}$, let ${\mathcal{M}\subseteq\real^{m\times n}}$ denote the set of gain matrices $\bK$ such that ${\bu=-\bK\bx}$ renders the origin GES and globally satisfies the HOCBF constraints~\eqref{eq:hocbf-linear}. Characterize a gain matrix ${\bK^\star\in\mathcal{M}}$ such that:
    $$
    J(\bx_0,\bx\mapsto-\bK^\star\bx) = \min_{\bK\in\mathcal M} J(\bx_0,\bx\mapsto-\bK\bx),~\forall \bx_0\in\real^n.
    $$
\end{problem}




For robustness, we consider matched disturbances $\bd\in\real^m$ in the dynamics: $\dot{\bx} = \bA\bx+\bB(\bu+\bd)$. With the feedback convention $\bu=-\bK\bx$, and assuming $\bA-\bB\bK$ is Hurwitz, the corresponding input sensitivity transfer matrix from control feedback, from $\bd$ to $\bu+\bd$, is:
\begin{equation}
\bS_{\bK}(s)
=
\bI-\bK(s\bI-(\bA-\bB\bK))^{-1}\bB .
\end{equation}

For SISO systems,  if $\norm{\bS_{\bK}}_{\infty} < \gamma$, 
the phase margin is lower bounded by $2 \arcsin(\frac{1}{2 \gamma})$ and the gain margin is lower bounded by $\frac{\gamma}{\gamma-1}$~\cite{SS-IP:05} (see~\cite{NL-NS-MA:81} for MIMO systems).

\begin{problem}\label{prob:problem-robustness}
    Let $\gamma > 1$. Characterize the existence of a gain matrix $\bK\in\mathcal{M}$ such that
    $\norm{\bS_{\bK}}_{\infty} < \gamma$.
\end{problem}

%
\section{Linear Controllers for HOCBF constraints}
In this section, we characterize the set of linear feedback controllers that globally satisfy the HOCBF constraints~\eqref{eq:hocbf-linear}. Here we focus only on safety. For each $i\in[p]$, we use the shorthand notation $\bell_i \triangleq \bB^\top (\bA^{r_i-1})^\top \bc_i$ for the constraint normals  in~\eqref{eq:hocbf-linear}. We begin our exposition with the case where the vectors $\{\bell_i\}_{i=1}^p$ are linearly independent, and the linearly dependent case is handled later.

\subsection{Linearly independent constraint normals}
When $\{\bell_i\}_{i=1}^p$ are linearly independent, the matrix $\bL\in\real^{m\times p}$ with columns $\{ \bell_i \}_{i=1}^p$ is full rank. Let $\bL_\perp\in\real^{m\times(m-p)}$ be a matrix whose columns are a basis for $\ker(\bL^\top)$ (cf.~\cite[Section 0.6.6]{RAH-CRJ:12}). Then we can decompose the control input $\bu = \bL \boldsymbol{\nu} + \bL_\perp\tilde{\bu}$, where $\boldsymbol{\nu}\in\real^p$, $\tilde{\bu}\in\real^{m-p}$. Under this decomposition, the HOCBF constraints~\eqref{eq:hocbf-linear} depend only on the $\boldsymbol{\nu}$ component as:
$$
\mu_i(\bx,\boldsymbol{\nu}) \triangleq \bc_i^\top \bPhi_i(\bA) \bx + \bell_i^\top \bL \boldsymbol{\nu} + \alpha_i d_i \geq 0.
$$
This reformulation leads to the following characterization.

\begin{proposition}\longthmtitle{Linear controllers satisfying linearly independent HOCBF constraints}\label{prop:linear-controllers-satisfying-linearly-independent-hocbf}
    Consider~\eqref{eq:linear-system} with a safe set $\Cc$ in~\eqref{eq:safety_constraint} defined by $\{h_i\}_{i=1}^p$. 
    Suppose that $\{\bell_i\}_{i=1}^p$ are linearly independent. Then there exists $\bK\in\real^{m\times n}$ such that $\bu=-\bK\bx$ satisfies all $p$ HOCBF constraints~\eqref{eq:hocbf-linear} for all $\bx\in\real^n$ if and only if $d_i\geq 0$ for all $i\in [p]$. In particular, the set of all such gain matrices is given by:
     $\setdef{\bK\in\real^{m\times n}}{ \bK = \bK_0 + \bL_\perp\tilde{\bK}, \tilde{\bK} \in \real^{(m-p)\times n} }$ where $\bK_0 = \bL(\bL^\top \bL)^{-1} \tilde{\bC}$ and $\tilde{\bC} = [\bc_1^\top \bPhi_1(\bA); \hdots; \bc_p^\top \bPhi_p(\bA)]$.
\end{proposition}
\begin{proof}
    Under feedback ${\bu = -\bK\bx}$, we have 
    $\mu_i = (\bc_i^\top \bPhi_i(\bA) - \bell_i^\top \bK)\bx + \alpha_i d_i$, which necessarily requires that 
    $\bar{\bv}^\top = \bc_i^\top \bPhi_i(\bA) - \bell_i^\top \bK = \mathbf{0}_n^\top$. Indeed, if this was not the case, we would have a point of the form $\bx = -t \bar{\bv}$, with $t > 0$ sufficiently large ($t > \frac{\alpha_i d_i}{\norm{\bar{\bv}}^2}$), for which $\mu_i < 0$, which would violate the $i$-th HOCBF constraint. As such, $d_i\geq 0$ is required for HOCBF constraints to hold.
    In addition, $\bK$ must be the solution to
    $\tilde{\bC} = \bL^\top \bK$.
    Now, because $\begin{bmatrix}\bL & \bL_\perp\end{bmatrix}$ is invertible, we may uniquely express any valid solution of $\tilde{\bC} = \bL^\top \bK$ as  $\bK = \bL \bK_p+\bL_\perp\tilde\bK$, with $\bK_p\in\real^{p\times n}$ and $\tilde\bK\in\real^{(m-p)\times n}$. Then, $\tilde{\bC} = \bL^\top \bK$ simplifies to $\tilde{\bC} = \bL^\top\bL \bK_p$, which implies that $\bK = \bK_0+\bL_\perp \tilde \bK$    for some $\tilde{\bK} \in \real^{(m-p)\times n}$, from where the result follows.
\end{proof}
Proposition~\ref{prop:linear-controllers-satisfying-linearly-independent-hocbf} characterizes the set of linear controllers that satisfy the HOCBF constraints~\eqref{eq:hocbf-linear} under the assumption that $\{ \boldsymbol{\ell}_i \}_{i=1}^p$ are linearly independent. Such controllers exists under the necessary and sufficient condition: $d_i\geq 0$ for all $i\in[p]$, which implies the origin is in the safe set~\eqref{eq:safety_constraint}.

\begin{remark}\longthmtitle{Origin in the safe set}
    {\rm 
    Since $h_i(\bx)= \bc_i^\top \bx+d_i$, the condition $d_i\geq 0$ is equivalent to $\bzero\in \Cc_i$. Moreover, because $\alpha_{i,j}> 0$ is positive for all $j\in[r_i]$, the condition is also equivalent to the origin being in the HOCBF safe set, i.e., $\mathbf{0}_n\in\bar{\Cc}_i$. Indeed, since we seek to stabilize the system to the origin while rendering $\bar{\Cc}$ forward invariant, the condition is necessary for the safe stabilization problem, and it often holds in practice. 
    }\demo
\end{remark}


\subsection{Linearly dependent constraint normals}
Here we consider $\bar p$ HOCBF constraints, where in general, the normals $\{ \bell_i \}_{i=1}^{\bar p}$ could be linearly dependent. Nevertheless, we assume without loss of generality that the first $p \leq \bar p$ vectors $\{ \bell_i \}_{i=1}^{p}$ are linearly independent, and $\{ \bell_i \}_{i=p+1}^{\bar p}$ are linear combinations of $\{ \bell_i \}_{i=1}^{p}$. 
The following result extends Proposition~\ref{prop:linear-controllers-satisfying-linearly-independent-hocbf} to this setting.

\begin{proposition}\longthmtitle{Linear controllers satisfying HOCBF constraints}\label{prop:linear-controllers-satisfying-HOCBF-general}
     Consider~\eqref{eq:linear-system} with safe set $\Cc$ in~\eqref{eq:safety_constraint} defined by $\{h_i\}_{i=1}^{\bar p}$. 
     There exists $\bK\in\real^{m\times n}$ such that $\bu=-\bK\bx$ satisfies all $\bar p$ HOCBF constraints~\eqref{eq:hocbf-linear} for all $\bx\in\real^n$ if and only if $d_i\geq 0$ for all $i\in [\bar p]$ and the compatibility requirement $\bc_i^\top\Phi_i(\bA)-\bell_i^\top \bK_0 = \bzero$ holds for all $i\in[\bar p]\setminus [p]$. In particular, the set of all such matrices is:
     $\setdef{\bK\in\real^{m\times n}}{ \bK = \bK_0 + \bL_\perp\tilde{\bK}, \tilde{\bK} \in \real^{(m- p)\times n} }$.
\end{proposition}
\begin{proof}
    By the same proof as in Proposition~\ref{prop:linear-controllers-satisfying-linearly-independent-hocbf}, 
    since $\{ \bell_i \}_{i=1}^{p}$ are linearly independent,
    the set of linear controllers satisfying the HOCBF conditions for $i\in[p]$
    is exactly as stated if and only if $d_i\geq 0$ for all $i\in[p]$.
    Now let $i\in[\bar p]\setminus[p]$. The same argument as in the proof of Proposition~\ref{prop:linear-controllers-satisfying-linearly-independent-hocbf} requires that $\bc_i^\top\Phi_i(\bA)-\bell_i^\top \bK = \bzero$. Then since $\bK=\bK_0 +\bL_\perp\tilde{\bK}$ is necessary and $\bell_i^\top \bL_\perp = \bzero$ from linear dependency, we have the compatibility requirement as stated. 
\end{proof}
Proposition~\ref{prop:linear-controllers-satisfying-HOCBF-general}
provides necessary and sufficient conditions for the existence of a linear feedback controller satisfying all HOCBF constraints globally. These involve the origin being in the safe set and a compatibility condition for linearly dependent constraints. Moreover, it completely characterizes the set of such controllers. This concludes our safety portion of the analysis.
\begin{remark}\longthmtitle{Box constraints}\label{rem:box-constraints}
    {\rm
    A common source of linearly dependent HOCBF constraints arises from box constraints on state variables. Consider a pair of constraints with indices $i_1, i_2 \in [\bar p]$
    such that $\bc_{i_2} = -\bc_{i_1}$ and $d_{i_1} = -d_{i,\min}$, $d_{i_2} = d_{i,\max}$.
    In this case, Proposition~\ref{prop:linear-controllers-satisfying-HOCBF-general} requires that $d_{i,\min} \leq 0 \leq d_{i,\max}$, i.e., the origin is in the box constraint. Moreover, the  requirement $\bc_{i_2}^\top \bPhi_{i_2}(\bA) - \bell_{i_2}^\top\bK_0 = \mathbf{0}_n^\top$  is automatically satisfied if $\alpha_{i_1,j} = \alpha_{i_2,j}$ for each $j\in[r_{i_1}]$.
    \demo
    }
\end{remark}

\section{HOCBF-Based Stabilizing Controllers}\label{sec:stabilizing-controllers-satisfying-HOCBF}
In this section we characterize the set of controllers that not only satisfy the HOCBF conditions~\eqref{eq:hocbf-linear} but also stabilize the origin for system~\eqref{eq:linear-system}. In order to cleanly derive the conditions, we follow an approach similar to that of~\cite{JJC-CJT-SS-KS:25} and write the dynamics~\eqref{eq:linear-system} in the \textit{CBF output form} (cf.~\cite[Definition 6, 7]{JJC-CJT-SS-KS:25}), which facilitates our ensuing analysis.

\subsection{Linear Dynamics in CBF Output Form}\label{sec:linear-dynamics-CBF-output-form}
Consider $\bar p$ HOCBF constraints~\eqref{eq:hocbf-linear} where the first $p$ constraints are linearly independent while the rest are linear combinations of them. The following is a key technical result for our coordinate transformation into CBF output form.

\begin{lemma}\label{lem:linear-independence-output-derivatives}
    Suppose $\{\bell_i\}_{i=1}^p$ are linearly independent. Then, the vectors $\mathcal{W} := \{ \bc_i, \bPhi_{i,1}(\bA)^\top \bc_i, \hdots, \bPhi_{i,r_i-1}(\bA)^\top \bc_i \}_{i=1}^p$
    are linearly independent.
\end{lemma}
\begin{proof}
    Since $\bL$ is full column rank, the matrix defined in~\cite[Equation 5.2]{AI:95}) is precisely $\bL^\top\bL$, which is nonsingular. Therefore, $\{ h_i \}_{i=1}^p$ has vector relative degree $(r_1, \hdots, r_p)$ for system $\dot{\bx} = \bA\bx + \bB\bL\boldsymbol{\nu}$. Then, by~\cite[Lemma 5.1.1]{AI:95}, the vectors $\bar{\mathcal{W}} := \{ \bc_i, \bA^\top \bc_i, \hdots, (\bA^\top)^{r_i-1} \bc_i \}_{i=1}^p$ are linearly independent. Now, with an invertible transformation from $\bar{\mathcal{W}}$ to $\mathcal{W}$ (cf.~\cite[Equation 15]{JJC-CJT-SS-KS:25}), the result follows.
\end{proof}

By Lemma~\ref{lem:linear-independence-output-derivatives}, the HOCBF auxiliary functions $h_{i,j}$ are linearly independent, so there exists an affine coordinate transformation from~\eqref{eq:linear-system} to the the coordinates defined by $h_{i,j}$ (along with some internal variables complementing them).
Since our goal is to study the stability of the origin, we will instead only use a linear (instead of affine) transformation. This results in shifted CBF output coordinates that preserve the origin as the desired equilibrium point.

Let $\bT_{\boldsymbol{\psi}}\in\real^{n \times r}$ be a full column rank matrix whose columns are the vectors in $\mathcal{W}$ (as defined in Lemma~\ref{lem:linear-independence-output-derivatives}) with $r:= \sum_{i=1}^p r_i \leq n$.
On the other hand, let $\bT_{\boldsymbol{\eta}} \in \real^{n\times(n-r)}$ be
a full column rank matrix such that $\bT_{\boldsymbol{\psi}}^\top \bT_{\boldsymbol{\eta}} = \mathbf{0}_{r\times(n-r)}$
(i.e., $\bT_{\boldsymbol{\eta}}$ is selected so that its columns span $\text{ker}(\bT_{\boldsymbol{\psi}}^\top)$).
This completes the linear transformation $\bx \mapsto (\bT_{\boldsymbol{\psi}}^\top \bx; \bT_{\boldsymbol{\eta}}^\top \bx) = (\boldsymbol{\psi}, \boldsymbol{\eta})$.
%
Here, $\bpsi\in\real^r$ is an aggregate of $\bpsi_i\in\real^{r_i}$ such that $\psi_{i,j} = \bc_i^\top \bPhi_{i,j}(\bA)\bx$ for $j\in\{0\}\cup[r_i-1]$. Then for $i\in[p]$ and $j\in[0]\cup[r_i-2]$, we have the dynamics:
\begin{align*}
    &\dot{\psi}_{i,j} \! = \! \dot{h}_{i,j} \!=\! h_{i,j+1}(\bx) \! - \! \alpha_{i,j+1}h_{i,j}(\bx)
    \! = \! \psi_{i,j+1} \! - \! \alpha_{i,j+1} \psi_{i,j}, \\
    &\dot{\psi}_{i,r_i-1}(\bx) = \dot{h}_{i,r_i-1}(\bx) = 
    \mu_i - \alpha_i d_i - \alpha_{i,r_i} \psi_{i,r_i-1},
\end{align*}
for which we can write compactly as:
\begin{align}\label{eq:psi-i-dynamics}
    \dot{\boldsymbol{\psi}}_i = \boldsymbol{\Psi}_i \boldsymbol{\psi}_i + \bB_{\bpsi_i}\bar{\boldsymbol{\mu}},
\end{align}
where $\bPsi_{i}\in\real^{r_i \times r_i}$ is 
such that $(\bPsi_{i})_{jj} = -\alpha_{i,j}$ for $j\in[r_i]$, $(\bPsi_i)_{j,j+1}=1$ for $j\in[r_i-1]$, and $(\bPsi_{i})_{jk} = 0$ for other $j,k\in[r_i]$.
The matrix $\bB_{\bpsi_i}$ has $1$ in its $(r_i, i)$-th entry and $0$ for all other entries.
We have also defined $\bar{\boldsymbol{\mu}} \triangleq [\mu_1 - \alpha_1 d_1; \hdots; \mu_p - \alpha_p d_p] \in \real^p$. 
Note that $\boldsymbol{\nu}$ can be expressed in terms of $\bar{\boldsymbol{\mu}}$ as $\boldsymbol{\nu} = (\bL^\top \bL)^{-1}( \bar{\boldsymbol{\mu}} - \tilde{\bC}\bx )$, with $\tilde{\bC} = [\bc_1^\top \bPhi_1(\bA); \hdots; \bc_p^\top \bPhi_p(\bA)]$.

Next, we construct the dynamics for the $\boldsymbol{\eta}$ coordinates. Noting that $\bT_{\boldsymbol{\psi}}$ and $\bT_{\boldsymbol{\eta}}$ are orthogonal and together form a basis, the inverse transformation is $\bx = \bT_{\boldsymbol{\psi}} (\bT_{\boldsymbol{\psi}}^\top \bT_{\boldsymbol{\psi}})^{-1} \bpsi + \bT_{\boldsymbol{\eta}} (\bT_{\boldsymbol{\eta}}^\top \bT_{\boldsymbol{\eta}})^{-1} \boldsymbol{\eta}$.
Together with $\dot{\boldsymbol{\eta}} = \bT_{\boldsymbol{\eta}}^\top \dot \bx$,
we derive:
\begin{subequations}
\begin{align}
    \dot{\boldsymbol{\psi}} &= \bPsi \bpsi + \bB_{\bpsi}\bar{\boldsymbol{\mu}},~\label{eq:psi-dynamics} \\
    \dot{\boldsymbol{\eta}} &= \boldsymbol{\Gamma}_{\bpsi} \bpsi 
    + \boldsymbol{\Gamma}_{\boldsymbol{\eta}} \boldsymbol{\eta}
    + \boldsymbol{\Gamma}_{\mu} \bar{\boldsymbol{\mu}}
    + \bT_{\boldsymbol{\eta}}^\top \bB \bL_\perp \tilde{\bu},~\label{eq:eta-dynamics}
\end{align}
\label{eq:dynamics-normal-form}
\end{subequations}
where: $\bPsi = \text{blkdiag}(\{ \bPsi_i \}_{i=1}^p), \quad \bB_{\psi} = [\bB_{\bpsi_1}; \hdots; \bB_{\bpsi_p}],$
\begin{align*}
    &\boldsymbol{\Gamma}_{\bpsi} = \hat{\bA} \bT_{\boldsymbol{\psi}} (\bT_{\boldsymbol{\psi}}^\top \bT_{\boldsymbol{\psi}})^{-1},  \boldsymbol{\Gamma}_{\boldsymbol{\eta}} = \hat{\bA} \bT_{\boldsymbol{\eta}} (\bT_{\boldsymbol{\eta}}^\top \bT_{\boldsymbol{\eta}})^{-1}, \\ &\boldsymbol{\Gamma}_{\boldsymbol{\mu}} = \bT_{\boldsymbol{\eta}}^\top \bB \bL ( \bL^\top \bL )^{-1}, 
    \hat{\bA} = \bT_{\boldsymbol{\eta}}^\top \bA - \bT_{\boldsymbol{\eta}}^\top \bB \bL ( \bL^\top \bL )^{-1} \tilde{\bC},  
\end{align*}

Note that since the maps $\bx \to (\bpsi,\boldsymbol{\eta})$ and $\bu \to (\boldsymbol{\nu},\tilde{\bu})$ are invertible, the stabilizability properties of~\eqref{eq:dynamics-normal-form} are equivalent to those of~\eqref{eq:linear-system}. This observation will be leveraged next.

\subsection{Stabilizability with HOCBF constraints}

Since Proposition~\ref{prop:linear-controllers-satisfying-HOCBF-general} characterizes the set of all gain matrices for which the HOCBF constraints~\eqref{eq:hocbf-linear}  hold globally, safe stabilization can be achieved by selecting a gain matrix $\bK$ from this set such that $\bA-\bB\bK$ is also Hurwitz. Here, we leverage the shifted CBF output form~\eqref{eq:dynamics-normal-form} to solve Problem~\ref{prob:problem-stability}.


\begin{theorem}\longthmtitle{Conditions for safe stabilizability with linear controllers}\label{thm:conditions-stabilizability}
     Consider~\eqref{eq:linear-system} with safe set in~\eqref{eq:safety_constraint} defined by $\{h_i\}_{i=1}^{\bar p}$. 
     Assume that $\{ \bell_i \}_{i=1}^{p}$ are linearly independent and $\{ \bell_i \}_{i=p+1}^{\bar p}$ are linear combinations of $\{ \bell_i \}_{i=1}^{p}$.
     There exists $\bK\in\real^{m\times n}$ 
     with $\bu=-\bK\bx$ satisfying all $\bar p$ HOCBF constraints~\eqref{eq:hocbf-linear} for all $\bx\in\real^n$ and renders $\mathbf{0}_n$ GES if and only if: 
     \begin{enumerate}
        \item\label{it:di-condition} $d_i\geq 0$ for all $i\in[\bar p]$;
        \item\label{it:compatibility-condition} $\bc_i^\top\Phi_i(\bA)-\bell_i^\top \bK_0 = \bzero_n^\top$ holds for all $i\in[\bar p]\setminus [p]$;
        \item $(\boldsymbol{\Gamma}_{\boldsymbol{\eta}}, \bT_{\boldsymbol{\eta}}^\top \bB \bL_\perp)$ is stabilizable (or equivalently, $(\bA-\bB\bK_0,\bB\bL_\perp)$ is stabilizable).
     \end{enumerate}
    In particular, the set of all such matrices is
     $
     \setdef{\bK\in\real^{m\times n}}{ \bK = \bK_0 + \bL_\perp\tilde{\bK},~\tilde{\bK} \in \real^{(m- p)\times n},~\bA-\bB\bK \textup{ Hurwitz}}.
     $
\end{theorem}
\begin{proof}
    When $\bu = -\bK\bx$, as shown in the proof of Proposition~\ref{prop:linear-controllers-satisfying-HOCBF-general} we have
    $\bar{\bv}^\top = \bc_i^\top \bPhi_i(\bA) - \bell_i^\top \bK = \mathbf{0}_n^\top$.
    Hence, we necessarily have $\mu_i = \alpha_i d_i$ for all $i\in[\bar p]$, and $\bar \bmu = \bzero$.
    Further, since $\bx\mapsto (\bpsi,\boldsymbol{\eta})$ and the input decomposition $\bu = \bL \boldsymbol{\nu} +\bL_\perp \tilde \bu$ are invertible linear mappings, there exist $\bK_1\in\real^{(m-p)\times r}$ and $\bK_2\in\real^{(m-p)\times (n-r)}$ such that $\tilde{\bu} = -\bar{\bK}_1\bpsi -\bar{\bK}_2\boldsymbol{\eta}$, and
    the dynamics~\eqref{eq:dynamics-normal-form} take the form: 
    \begin{align}\label{eq:closed-loop-normal-form}
        \begin{pmatrix}
            \dot{\bpsi} \\
            \dot{\boldsymbol{\eta}}
        \end{pmatrix} = 
        \underbrace{
        \begin{pmatrix}
            \bPsi & \mathbf{0}_{r\times(n-r)} \\
            \boldsymbol{\Gamma}_{\bpsi} - \bT_{\boldsymbol{\eta}}^\top \bB \bL_\perp \bar{\bK}_1 & \boldsymbol{\Gamma}_{\boldsymbol{\eta}} - \bT_{\boldsymbol{\eta}}^\top \bB \bL_\perp \bar{\bK}_2
        \end{pmatrix}
        }_{\boldsymbol{\Gamma}}
        \begin{pmatrix}
            \bpsi \\ \boldsymbol{\eta}
        \end{pmatrix}.
    \end{align}
    Note that $\bPsi$ is Hurwitz (because its diagonal blocks~$\bPsi_i$ are Hurwitz).
    Hence, $\boldsymbol{\Gamma}$ is Hurwitz if and only if there exists $\bar{\bK}_2$ such that $\boldsymbol{\Gamma}_{\boldsymbol{\eta}} - \bT_{\boldsymbol{\eta}}^\top \bB \bL_\perp \bar{\bK}_2$ is Hurwitz. 
    This occurs if and only if $(\boldsymbol{\Gamma}_{\boldsymbol{\eta}}, \bT_{\boldsymbol{\eta}}^\top \bB \bL_\perp)$ is stabilizable.
    Finally, note that since $\bx \mapsto (\bpsi,\boldsymbol{\eta})$ and $\bu\mapsto(\boldsymbol{\nu},\tilde{\bu})$ are invertible, $(\boldsymbol{\Gamma}_{\boldsymbol{\eta}},\bT_{\boldsymbol{\eta}}^\top \bB\bL_\perp)$ is stabilizable if and only if there exists $\tilde{\bK}$ such that $\bA - \bB(\bK_0 + \bL_\perp\tilde{\bK})$ is Hurwitz, which is equivalent to $(\bA-\bB\bK_0,\bB\bL_\perp)$ being stabilizable.
\end{proof}

Theorem~\ref{thm:conditions-stabilizability} provides a characterization of the set of linear stabilizing controllers that satisfy~\eqref{eq:hocbf-linear}. It also provides necessary and sufficient conditions under which this set is non-empty.
In the case $m=p$ (if~\ref{it:di-condition} and~\ref{it:compatibility-condition} hold), it ensures that $\bu = -\bK_0 \bx$ is the only linear controller satisfying~\eqref{eq:hocbf-linear} globally, and that it is stabilizing if and only if $\bGamma_{\boldsymbol{\eta}}$ is Hurwitz.

\begin{remark}\longthmtitle{Nonlinear controllers}\label{rem:nonlinear-controllers}
    {\rm
    When conditions in Theorem~\ref{thm:conditions-stabilizability} do not hold, nonlinear safe stabilizing controllers may still exist.
    In fact,~\cite[Theorem 3.6]{LQT-MKC:14} characterizes the exact conditions under which system~\eqref{eq:dynamics-normal-form} (with the constraint $\mu_i \geq 0$ for all $i\in[p]$) is stabilizable (with potentially nonlinear locally integrable open-loop control signals).
    Unfortunately, a constructive method to obtain such controller is not provided. 
    However, the conditions in~\cite[Theorem 3.6]{LQT-MKC:14} can be used to rule out the stability of 
    CBF-based safety filters (cf.~\cite{PM-YC-EDA-JC:26-jnls,PM-SSM-ADA:26}).
    }
    \demo
\end{remark}

\section{Optimal and Margin-Certified  Controllers}

We now address Problems~\ref{prob:optimality} and~\ref{prob:problem-robustness}. 
The following result shows that Problem~\ref{prob:optimality} can be solved through a single ARE. 

\begin{theorem}\longthmtitle{Optimal safe and stable linear controller}
\label{thm:optimal-globally-safe-linear-controller}
Suppose that the necessary and sufficient conditions for linear safe stabilization in Theorem~\ref{thm:conditions-stabilizability} hold and that $p<m$.
Given $\bQ \in \mathbb{S}_{++}^n$, $\bR \in \mathbb{S}_{++}^m$, define
$\bA_s = \bA-\bB\bK_0$, $\bB_s = \bB \bL_\perp$, $\bQ_s=\bQ+\bK_0^\top\bR\bK_0$, $\bS_s=-\bK_0^\top \bR \bL_\perp$, $\bR_s:=\bL_\perp^\top\bR\bL_\perp$. 
Let $\bP=\bP^\top\succ0$ denote the unique solution
of the generalized ARE:
\[
    \bA_s^\top\bP+\bP\bA_s
    -
    (\bP\bB_s+\bS_s)\bR_s^{-1}
    (\bB_s^\top\bP+\bS_s^\top)
    +
    \bQ_s
    =
    \mathbf{0}_{n\times n}.
\]
Then, the solution for Problem~\ref{prob:optimality} is given by:
\[
    \bK^\star
    =
    \bK_0
    +
    \bL_\perp\bR_s^{-1}
    (\bB_s^\top\bP+\bS_s^\top).
\]
\end{theorem}

\begin{proof}
From Theorem~\ref{thm:conditions-stabilizability}, the feedback is given by $\bu=-\bK_0\bx+\bL_\perp\tilde{\bu}$,
where $\tilde{\bu}\in\real^{m-p}$ is a new linear control input, rewriting the cost function as: 
\begin{align*}
    \bx^\top\bQ\bx+\bu^\top\bR\bu = \begin{bmatrix}\bx^\top & \tilde{\bu}^\top \end{bmatrix}
    \underbrace{\begin{bmatrix}
        \bQ_s & \bS_s \\
        \bS_s^\top & \bR_s
    \end{bmatrix}}_{\bM_s} \begin{bmatrix}
        \bx \\ \tilde \bu
    \end{bmatrix}.
\end{align*}
Since $\bR\succ0$ and $\bL_\perp$ has full column rank, 
$\bR_s \succ 0$.
Thus, Problem~\ref{prob:optimality} is a continuous-time LQR problem for
$\dot{\bx}=\bA_s\bx+\bB_s\tilde{\bu}$, with a state-control cross
term~\cite[Section 3.4]{BDOA-JBM:90}. Moreover,
Theorem~\ref{thm:conditions-stabilizability} implies that
$(\bA_s,\bB_s)$ is stabilizable. Since the original running cost is
positive definite in $(\bx,\tilde{\bu})$, we have $\bM_s\succ0$, and therefore,
by the Schur complement,
$\bQ_s-\bS_s\bR_s^{-1}\bS_s^\top\succ0$. Hence, standard generalized
LQR theory guarantees a unique stabilizing solution
$\bP=\bP^\top\succ0$ of the generalized ARE, yielding
$\tilde{\bu}^\star=-\bR_s^{-1}(\bB_s^\top\bP+\bS_s^\top)\bx$.
Finally we have ${\bu} = -\bK_0 \bx + \bL_\perp \tilde{\bu}^\star = -\bK^\star \bx$.
\end{proof}



Theorem~\ref{thm:optimal-globally-safe-linear-controller} gives the LQR-optimal controller within the class of globally safe linear feedback laws. 
Unfortunately, since the proof relies on solving an ARE for a modified system with input matrix $\bB\bL_\perp$ (instead of $\bB$), the standard phase and gain margins associated with LQR controllers (cf.~\cite[Section 14.4]{KZ-JD-KG:95}) do not necessarily hold, since the input disturbance could enter through channels outside the span of $\bL_\perp$.
Next, we design a controller that globally satisfies the HOCBF constraints
and has certified gain and phase margins, i.e., solving
Problem~\ref{prob:problem-robustness}.
%
The following result gives a convex synthesis condition for such robust controller.

\begin{proposition}\longthmtitle{Safe controller with certified stability margins}\label{prop:robust-safe-linear-synthesis}
Suppose that the necessary and sufficient conditions for linear safe stabilization in Theorem~\ref{thm:conditions-stabilizability} hold.
For a given $\gamma>1$, Problem~\ref{prob:problem-robustness} admits a solution if and only if there exist $\bX \in \mathbb{S}_{++}^n$, $\mathbf{Y}\in\real^{m\times n}$ such that:
\begin{subequations}
\begin{equation}
\begin{bmatrix}
\bA\mathbf{X}+\mathbf{X}\bA^\top-\bB\mathbf{Y}-\mathbf{Y}^\top\bB^\top 
& \bB 
& -\mathbf{Y}^\top\\
\bB^\top 
& -\gamma\bI_m 
& \bI_m\\
-\mathbf{Y} 
& \bI_m 
& -\gamma\bI_m
\end{bmatrix}
\prec 0,
\label{eq:robust-safe-lmi}
\end{equation}
\begin{equation}
\bL^{\top}\mathbf{Y} = \tilde{\bC}\mathbf{X}.
\label{eq:robust-safe-equality}
\end{equation}
\end{subequations}
In which case, $\bK=\bY\bX^{-1}$ solves Problem~\ref{prob:problem-robustness}.
\end{proposition}
\begin{proof}
For $\bK=\bY\bX^{-1}$, let $\bA_{\bK}=\bA-\bB\bK$ and
${\bP}=\bX^{-1}/\gamma\succ0$. A congruence
transformation of~\eqref{eq:robust-safe-lmi} with
$\operatorname{blkdiag}(\bX^{-1},\bI_m,\bI_m)$, followed by
division by $\gamma$, gives
\[
\begin{bmatrix}
\bA_{\bK}^\top{\bP}+{\bP}\bA_{\bK}
& {\bP}\bB & -\bK^\top/\gamma\\
\bB^\top{\bP} & -\bI_m & \bI_m/\gamma\\
-\bK/\gamma & \bI_m/\gamma & -\bI_m
\end{bmatrix}\prec0.
\]
Taking the Schur complement of the lower-right block gives the
strict bounded-real LMI~\cite[Sec.~2.7.3]{SB-LEG-EF-VB:94}
for the realization
$(\bA_{\bK},\bB,-\bK/\gamma,\bI_m/\gamma)$ of
$\gamma^{-1}\bS_{\bK}$. Therefore, $\bA_{\bK}$ is Hurwitz and
$\norm{\bS_{\bK}}_\infty<\gamma$.
Under the assumptions of
Theorem~\ref{thm:conditions-stabilizability},
Proposition~\ref{prop:linear-controllers-satisfying-HOCBF-general}
implies that $\bu=-\bK\bx$ satisfies all HOCBF constraints globally
if and only if $\bL^\top\bK=\tilde{\bC}$. Since
$\bY=\bK\bX$, this is equivalent to
$\bL^\top\bY=\tilde{\bC}\bX$.

Conversely, suppose that Problem~\ref{prob:problem-robustness}
admits a solution $\bK\in\mathcal{M}$. By the strict bounded-real
lemma, there exists ${\bP}\succ0$ satisfying the LMI above.
Setting $\bX=(\gamma{\bP})^{-1}$ and $\bY=\bK\bX$ and
reversing the preceding transformations yields
\eqref{eq:robust-safe-lmi}. Moreover,
$\bL^\top\bK=\tilde{\bC}$ implies
\eqref{eq:robust-safe-equality}.
\end{proof}

Thus, a safe controller with certified stability margins can be computed
by minimizing $\gamma$ subject to~\eqref{eq:robust-safe-lmi}
and~\eqref{eq:robust-safe-equality}. The equality constraint enforces
nominal HOCBF safety, while the LMI certifies robust stability margins.
Alternatively, one can also fix a desired value of $\gamma$ and solve the feasibility problem of finding $\bX, \bY$ satisfying~\eqref{eq:robust-safe-lmi} and \eqref{eq:robust-safe-equality}.
\color{black}

\section{Numerical Example}

Here we apply the results of the paper to the roll-yaw dynamics of a mid-size aircraft around an operating point with velocity $717.17$ ft/sec, altitude 25000 ft and angle of attack $4.5627^\circ$~\cite[Section 14.8]{EL-KAW:24}.
The state is $\bx = [\beta, p_s, r_s]^\top$, where $\beta$ is sideslip (rad), and $p_s$, $r_s$ are roll and yaw rates (rad/s). The inputs are aileron and rudder deflections $\delta_a$ and $\delta_r$ (rad).
The dynamics are linear as in~\cite[Section 5.2]{PM-SSM-PO-LY-ED-JWB-ADA:26}.
\normalsize
We consider constraints on the roll rate $p_s$ of the form $\bc_1 = [0, -1, 0]$, $d_1 = 0.4$, $\bc_2 = [0, 1, 0]$, $d_2 = 0.4$.
Our goal is to regulate $p_s$ to a desired piecewise constant commanded signal $y_{\text{cmd}}:\real_{\geq0}\to\real$. To do so while respecting the safety constraints we project $y_{\text{cmd}}$ onto the safe set.
Given that there are multiple solutions to the equilibrium equation $\bA \bx_\star + \bB \bu_\star = \mathbf{0}_3$, with $\bx_\star = [x_{1,\star}, \bar{y}_{\text{cmd}}(t), x_{3,\star}]$ and $x_{1,\star}, x_{3,\star}, \bu_\star$ unknowns, we select the one with minimum norm.
In this case, $r_1 = 1$ and therefore the CBF zero dynamics in~\eqref{eq:dynamics-normal-form} are two-dimensional $(\boldsymbol{\eta}\in\real^2)$ with a one-dimensional control input ($\tilde{u}\in\real$).
In this case, the matrix $\bGamma_{\boldsymbol{\eta}}$ is Hurwitz and therefore the pair $(\bGamma_{\boldsymbol{\eta}}, \bT_{\boldsymbol{\eta}}^\top \bB \bL_\perp)$ is stabilizable, which by Theorem~\ref{thm:conditions-stabilizability} means that the set of linear controllers that globally exponentially stabilize the origin and satisfy the CBF condition is non-empty.
We implement the LQR problem in Theorem~\ref{thm:optimal-globally-safe-linear-controller} with $\bQ = \bI_3$ and $\bR = \bI_2$.
Figure~\ref{fig:aircraft-example-lqr} showcases the evolution of the state variables for a reference tracking task for this LQR controller.
\begin{figure}
    \centering
    \includegraphics[width=0.89\linewidth]{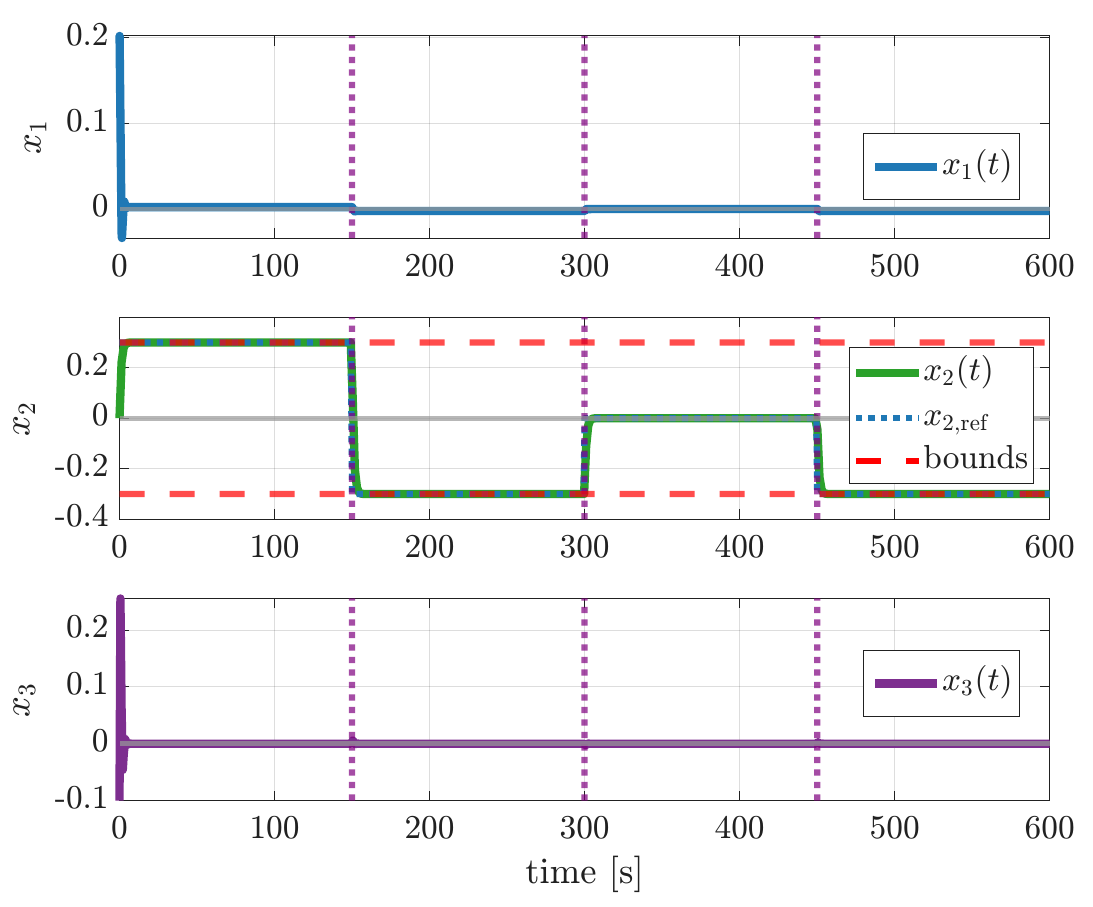}
    \caption{
    \footnotesize
    Reference tracking for roll-yaw dynamics example with LQR controller $\bK^\star = [ 0.8293, 0.0506, -0.3042; -0.1506, -0.0171, -0.8197]$.}
    \label{fig:aircraft-example-lqr}
\end{figure}
We also implement the robust controller from Proposition~\ref{prop:robust-safe-linear-synthesis} with $\gamma = 5$.
%
%
Although an upper bound on $\norm{\bS_{\bK}}_{\infty}$ of $\gamma = 5$ only guarantees gain margin in an interval $[0.83, 1.25]$ and a minimum phase margin of $11.48^\circ$ according to the formulas detailed in Section~\ref{sec:optimal-robust-linear-safe}, Figure~\ref{fig:aircraft_robust_margins} shows that in practice the gain and phase margins are up to $2$ and $60^\circ$ respectively.

\begin{figure}
    \centering
    \includegraphics[width=0.99\linewidth]{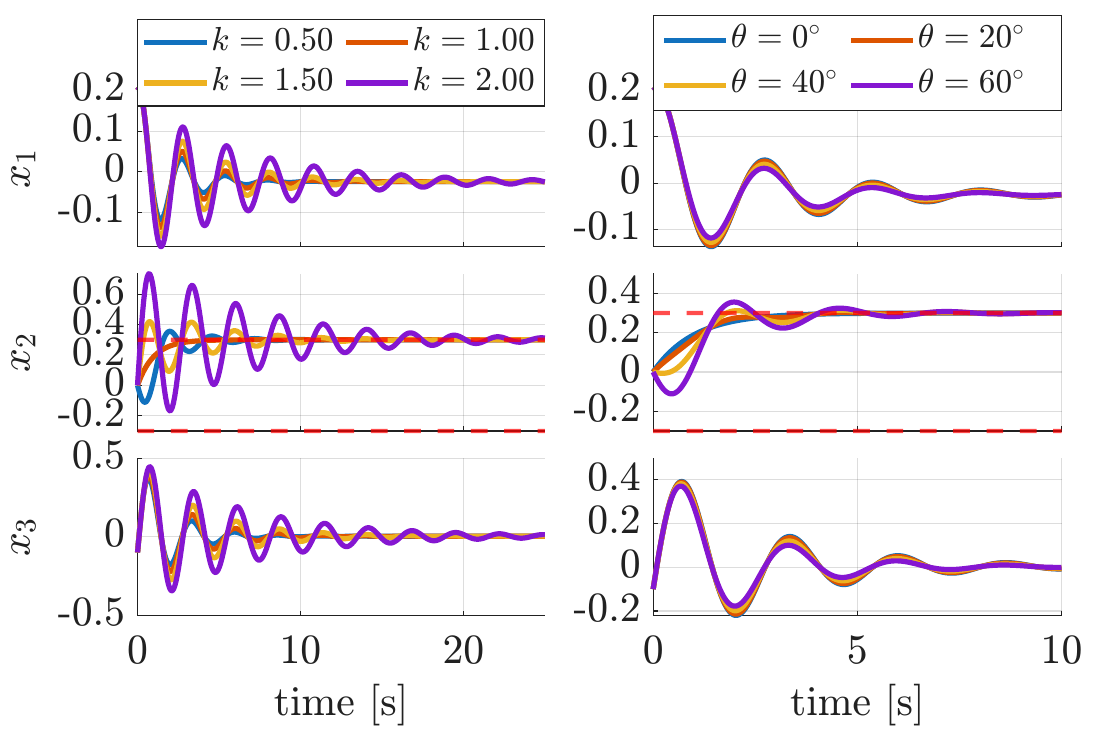}
    \caption{
    \footnotesize
    Gain (left) and phase (right) margin illustrations for the reference tracking roll-yaw dynamics example with the robust controller from Proposition~\ref{prop:robust-safe-linear-synthesis}, which is $\bK = [0.8732, 0.0551, -0.0674; -0.0205, -0.0037, -0.1176]$.}
    \label{fig:aircraft_robust_margins}
    \vspace{-15pt}
\end{figure}

\section{Conclusions}
We have studied the existence of linear controllers achieving simultaneous satisfaction of a set of HOCBF constraints and stabilization of a linear plant. By leveraging the concept of CBF output dynamics, we have provided an explicit characterization of such class of controllers, and necessary and sufficient conditions under which it is non-empty. Additionally, we have shown how to solve LQR and robust control problems within this class of controllers. 
Future work will seek to extend these ideas to nonlinear systems.

\bibliography{bib/alias,bib/Main-add,bib/Main,bib/JC,bib/PM,bib/New}
\bibliographystyle{IEEEtran}

\end{document}